\documentclass[english,12pt,a4paper,twoside,reqno]{article}
\usepackage{hyperref}
\usepackage[noadjust]{cite}
\usepackage{amsfonts,amsthm}
\renewcommand{\geq}{\geqslant}
\renewcommand{\leq}{\leqslant}

\usepackage{graphicx}
\usepackage{amsmath}
\usepackage{amssymb}
\usepackage{mathrsfs}
\usepackage[T1]{fontenc}
\usepackage[english]{babel}
\usepackage{enumerate}
\usepackage{euscript}
\numberwithin{equation}{section}
\usepackage{a4wide}
\usepackage{cite}        %
\title{A maximin characterisation of the escape rate of non-expansive mappings in metrically convex spaces}
\author{St\'ephane Gaubert\footnote{INRIA Saclay \& Centre de
Math\'ematiques Appliqu\'ees (CMAP), \'Ecole Polytechnique, 91128
Palaiseau, France. email: Stephane.Gaubert@inria.fr, fax: +33 1 69 33 46 46} \and Guillaume
Vigeral\footnote{Universit\'e Paris-Dauphine, CEREMADE, Place du Mar{\'e}chal De Lattre de Tassigny. 75775 Paris cedex 16, France. email: guillaumevigeral@gmail.com}}
\date{August 1, 2011%
}
\newcommand{\trace}{\operatorname{tr}}
\newcommand{\basepoint}[1]{\bar{#1}}
\newcommand{\centerpoint}[1]{{#1}^\circ}
\newtheorem{theorem}{Theorem}%
\newtheorem{proposition}[theorem]{Proposition}
\newtheorem{lem}[theorem]{Lemma}
\newtheorem{corollary}[theorem]{Corollary}

\theoremstyle{definition}
\newtheorem{defi}[theorem]{Definition}
\theoremstyle{remark}
\newtheorem{remark}[theorem]{Remark}
\newtheorem{example}[theorem]{Example}
\newcommand{\Funk}{\operatorname{RFunk}}
\newcommand{\Extr}{\operatorname{Extr}}
\newcommand{\Int}{\operatorname{int}}
\newcommand{\closure}{\operatorname{clo}}

\newcommand{\Spec}{\operatorname{Spec}}
\newcommand{\sH}{\mathscr{H}}
\newcommand{\sM}{\mathscr{M}}

\newcommand{\RR}{\mathbb{R}}

\def\<#1,#2>{\langle #1,#2\rangle}
\def\s(#1,#2){\<#1,#2>}
\newcommand{\T}{\mathbb{T}}
\newcommand{\Q}{\mathbb{Q}}
\newcommand{\NN}{\mathbb{N}}
\begin{document}
\maketitle
\renewcommand{\thefootnote}{}
\footnotetext{This work was performed when the second author was with INRIA Saclay -- \^Ile-de-France and CMAP, \'Ecole Polytechnique, being supported by the Digiteo project DIM08 PASO number 3389. The first author was partially supported by the Arpege programme of the French National Agency of Research (ANR), project ``ASOPT'', number ANR-08-SEGI-005.}
\renewcommand{\thefootnote}{\footnotemark[\value{footnote}]}

\begin{abstract}
We establish a maximin characterisation of the linear escape rate
of the orbits of a non-expansive mapping on a complete (hemi-)metric space, under a mild form of Busemann's non-positive curvature condition
(we require a distinguished family of geodesics with a common
origin to satisfy a convexity inequality).
This characterisation, which involves horofunctions, generalises
the Collatz-Wielandt characterisation of the
spectral radius of a non-negative matrix.
It yields as corollaries a theorem of Kohlberg and Neyman (1981), concerning non-expansive maps in Banach spaces, a variant of a Denjoy-Wolff type theorem of Karlsson (2001), together with a refinement of a theorem of Gunawardena and Walsh (2003), concerning order-preserving positively homogeneous self-maps of symmetric cones. An application to zero-sum stochastic games is also given.
\end{abstract}

\section{Introduction}
A self-map $T$ of a metric space $(X,d)$ is {\em non-expansive} if
$d(T(x),T(y))\leq d(x,y)$. A general problem consists
in studying the asymptotic behaviour of the orbits of $T$.
This is motivated in particular by the celebrated
theorem of Denjoy and Wolff~\cite{denjoy,wolff1,wolff2} on the iteration of holomorphic self-maps of the unit disk (these maps are non-expansive
in Poincare's metric).
Other motivations arise from the cases of non-expansive self-maps of Banach spaces
and of self-maps of cones that are non-expansive in Hilbert's, Thompson's or Riemannian metric, which have received a considerable
attention, due in particular to applications in game theory~\cite{rosenbergsorin,NeSo}, discrete event systems~\cite{mxpnxp0}, quadratic optimal control or filtering~\cite{bougerol}, and non-linear Perron-Frobenius theory~\cite{Nu}.
Several Denjoy-Wolff type results, either in the setting of metric spaces,
or concerning the special case of cones
have appeared, see in particular~\cite{beardon,Ka,AGLN,nussbaum07,lins}.
The reader may consult the monograph~\cite{reichshoiket} for an overview of the field.

In this paper, we consider the linear {\em escape rate}
\[
\rho(T)= \lim_{k\to\infty}\frac{d(x,T^k(x))}{k} \enspace.
\]
The latter always exists, by a standard subadditive argument.
Our main result is the following maximin
characterisation of the escape rate.
\begin{theorem}\label{th-minimax}
Let $T$ be a non-expansive self-map of a complete metrically star shaped
hemi-metric space $(X,d)$. Then,
\begin{align}\label{e-minimax}
\inf_{y\in X}d(y,T(y)) = \rho(T) = \max_{h}\inf_{x\in X}h(T(x))-h(x)  \enspace,
\end{align}
where the maximum, which is attained, is taken over the set of Martin
functions $h$ of $(X,d)$. If in addition the hemi-metric $d$ is bounded
from below (in particular, if $d$ is a metric),
and if $\rho(T)>0$, then any function attaining the maximum is a horofunction.
\end{theorem}
Let us explain the terminology.
{\em Hemi-metrics} are
analogous to metrics, but $d(x,y)$ and $d(y,x)$ may differ.
Their definition is a variation of that of {\em weak metrics}
in~\cite{papadopoulos}.
The notion of {\em metrically star shaped} space is the object
of Definition~\ref{defistarshaped} below.
It requires a distinguished family of geodesics
with a common origin to satisfy a convexity inequality. Hence,
it is a mild form of Busemann's classical non-positive curvature
condition~\cite{Papa05}.
{\em Martin functions} and {\em horofunctions} are special
Lipschitz functions of constant $1$ which arise when compactifying
the hemi-metric space $(X,d)$.
The set of horofunctions (the {\em horoboundary})
provides an abstract boundary of $X$.
See~\S\ref{subsec-horo} for details.

Theorem~\ref{th-minimax} is inspired by
the classical Collatz-Wielandt characterisation
of the Perron root $\rho(M)$ of a $n\times n$ non-negative matrix $M$. The latter shows
that
\begin{align}
\inf_{y\in \Int\RR_+^n}\max_{1\leq i\leq n}\frac{(My)_i}{y_i}
= \rho(M) = \max_{\scriptstyle u\in \RR_+^n\atop \scriptstyle u\neq 0} \min_{\scriptstyle 1\leq i\leq n\atop \scriptstyle  u_i \neq 0}
\frac{(Mu)_i}{u_i} \enspace .
\label{e-cw}
\end{align}
Here, $\RR_+$ denotes the set of real non-negative numbers, and so
$\RR_+^n$ is the standard positive cone.
Nussbaum showed in~\cite{nussbaum86} that the latter formula holds more generally when the map
$y\mapsto My$ is replaced by a non-linear continuous order-preserving $T$
of the standard positive cone.
When $T$ preserves the interior
of this cone, it is
non-expansive in
the (reverse) {\em Funk hemi-metric}~\cite{papadopoulos,walsh}
defined on this interior by
 \[
\Funk(x,y):= \log\sup_{%
1\leq i\leq n
}\frac{y_i}{x_i}\enspace .
\]
Thus,
\[
\Funk(T(x),T(y)) \leq \Funk(x,y) \enspace ,\qquad \forall x,y\in \Int \RR_+^n \enspace .
\]
The Collatz-Wielandt formula and its non-linear extension in~\cite{nussbaum86}
will be shown to be special instances of the maximin formula~\eqref{e-minimax}.

Theorem~\ref{th-minimax} is also motivated by the case of a Banach space $X$.
Then, %
one may consider the limit
\begin{align}\label{e-limit}
\lim_{k\to\infty} T^k(x)/k \enspace .
\end{align}
For instance, in the setting of zero-sum games, $T$ is the dynamic programming
or {\em Shapley operator}~\cite{shapley} (which acts on a Banach space of continuous functions equipped with the sup-norm), and the latter expression represents the limit of the mean payoff per time unit as a function of the initial position, when the horizon of the game tends to infinity, see \S\ref{appli-shapley} for more background.

A theorem of Pazy~\cite{pazy} shows that the limit~\eqref{e-limit} does exist
if $X$ is a Hilbert space; further improvements, under suitable strict convexity assumptions were done by Reich\cite{reich} and then by Kohlberg and Neyman~\cite{KoNe}:
the limit exists in the weak (resp.\ strong) topology if $X$ is reflexive and strictly convex (resp.\ if the norm of the dual space $X^\star$ is Fr\'echet differentiable).
Without strict convexity, the limit may not exist, even when $X$ is finite dimensional, however, a general result of Kohlberg and Neyman~\cite{KoNe}
(Corollary~\ref{linear} below) shows
that there is always a linear form $\varphi\in X^\star$ of norm $1$ such that
\[
\varphi(T^k(x))\geq \varphi(x)+k\overline{\rho}(T)
\]
holds for all $k\in\mathbb{N}$, where
\[
\overline{\rho}(T):= \inf_{y\in X}d(y,T(y))
\]
is precisely the term at the left-hand side of~\eqref{e-minimax}.
The theorem of Kohlberg and Neyman may also be thought of as
a special case of Theorem~\ref{th-minimax}. Indeed, any horofunction
$h$ attaining the maximum in~\eqref{e-minimax} satisfies
\begin{align}\label{e-ourineq}
h(T^k(x))\geq h(x)+k\overline{\rho}(T) \enspace, \forall x\in X
\end{align}
for all $k\in\mathbb{N}$. We shall see that the Kohlberg-Neyman theorem follows readily from this result. We shall also recover as corollaries a generalisation
of a result of Gaubert and Gunawardena~\cite{GG04}
(Corollary~\ref{linearcompact}) concerning Shapley operators, as well as a refinement of a result
 of Gunawardena and Walsh~\cite{gunawardenawalsh} valid in symmetric cones,
Corollary~\ref{cor-symm} below.

Theorem~\ref{th-minimax} turns out to be related to a Denjoy-Wolff type theorem of Karlsson~\cite{Ka}, Theorem~\ref{thmKarlsson} below. Karlsson showed
that an inequality similar to~\eqref{e-ourineq} holds.
He used a general subadditivity
argument, which does not require any non-positive curvature condition.
The statement of~\cite{Ka} assumes the metric space to be proper,
i.e., closed balls to be compact, but this assumption can be relaxed
by defining horofunctions with respect to the topology of pointwise convergence as we do here.
However,
there are essential discrepancies between~\eqref{e-ourineq}
and the result of~\cite{Ka}: the term $\overline{\rho}(T)$ is replaced there
by the linear escape rate $\rho(T)$, and the horofunction
$h$ {\em depends} on the choice of the point $x$, whereas it does
not in~\eqref{e-ourineq}. We shall
see that without a non-positive curvature condition, the
formula~\eqref{e-ourineq} may not hold
(we give an example in which $\overline{\rho}(T)>\rho(T)$
and $h$ necessarily depends on $x$). Hence, the
maximin Theorem~\ref{th-minimax} holds only under
more restrictive circumstances than the
Denjoy-Wolff type result of Karlsson. By comparison,
the interest of the more special Theorem~\ref{th-minimax}
lies in its strong duality nature: it allows one to
certify that $\rho(T)<\alpha$, or dually,
that $\rho(T)>\beta$, by exhibiting an element $y\in X$
such that $d(y,T(y))<\alpha$, or dually a Martin function $h$
such that $\inf_{x\in X}h(T(x))-h(x)> \beta$.
It should
be noted in this respect that when $X$ is a Banach space with a polyhedral norm,
the horoboundary admits a simple effective description, see~\cite{MR2268510,walsh07,AGW}.

\section{Definitions and preliminary results}
\subsection{Metrically star-shaped spaces}\label{subsec-hemi}
\begin{defi}
We say that $\delta:X\times X\rightarrow\mathbb{R}$ is a
{\em hemi-metric} on a set $X$ if the two following conditions are
satisfied for all $(x,y,z)\in X^3$:
\begin{enumerate}
\item\ $\delta(x,z)\leq \delta(x,y)+\delta(y,z)$
\item\ $\delta(x,y)=\delta(y,x)=0$ if and only if $x=y$.
\end{enumerate}
We then say that $(X,\delta)$ is a hemi-metric space.
\end{defi}
Notice that a hemi-metric is generally not a metric, since it is
neither symmetric nor non-negative.
This definition is closely related to the one of a {\em weak metric}
in~\cite{papadopoulos}, in which $\delta$ is required to be non-negative.
The last condition of the definition corresponds to the {\em weak separation} condition in the latter reference. We allow $\delta$ to take
negative values in order to deal with order-preserving
positively homogeneous self-maps of cones (Section~\ref{sec-cones}).

To any hemi-metric,
one can canonically associate a metric by the following
lemma.

\begin{lem}
For any hemi-metric $\delta$, the function
$d(x,y)=\max(\delta(x,y),\delta(y,x))$ is a metric on $X$.
\end{lem}

\begin{proof}
Verifying that $d$ is symmetric and satisfies the triangular
inequality is easy. The positivity of $d$ comes from the fact that
for all $x$ and $y$, $0=\delta(x,x)\leq\delta(x,y)+\delta(y,x)$.
Finally, if $d(x,y)=0$ then both $\delta(x,y)$ and $\delta(y,x)$ are
non-positive so by the same argument they are both null, thus $x=y$.
\end{proof}

In the sequel, $X$ is equipped with the topology induced by the
metric $d$. We shall say that $(X,\delta)$ is {\em complete} when the
associated metric space $(X,d)$ is complete.
\begin{defi}
A geodesic joining a point $x\in X$ to a point $y\in X$ is a map
$\gamma:[0,1]\rightarrow X$ such that $\gamma(0)=x$, $\gamma(1)=y$,
and such that for all $0\leq s\leq t\leq 1$,
\[
\delta(\gamma(s),\gamma(t))=(t-s)\delta(x,y).
\]
\end{defi}

\begin{defi}\label{defistarshaped}
We say that $(X,\delta)$ is {\em metrically star-shaped} with centre
$\centerpoint{x}$ if there exists a family of geodesics $\{\gamma_y\}_{y\in
X}$, such that $\gamma_y$ joins the centre $\centerpoint{x}$ to the point
$y$, and such that the following inequality is satisfied for every
$(y,z)\in X^2$ and $s\in[0,1]$:
\begin{equation}\label{eqdefstarshaped}
\delta\left(\gamma_y(s),\gamma_z(s)\right)\leq s\delta(y,z).
\end{equation}
If the hemi-metric $\delta$ is not a metric, we also require that for any $y$, the quantity $\delta(y,\gamma_y(s))$ tends to 0 as $s$ goes to 1.
\end{defi}
The condition~\eqref{eqdefstarshaped}
is a form of {\em metric convexity}~\cite{Papa05}.
In particular any Busemann space~\cite{Papa05} is metrically
star-shaped, but our definition is less demanding since we only
require the inequality~\eqref{eqdefstarshaped} to be satisfied for
one specific choice of geodesics.

\begin{example}
Any Banach space $(X,\|\cdot\|)$ is metrically star-shaped with
respect to any centre $\centerpoint{x}$ : it suffices to take the straight
lines as geodesics, i.e., for
any choice of centre $\centerpoint{x}$, the choice of
$\gamma_y(s)=\centerpoint{x}+s(y-\centerpoint{x})$ yields a metrically
star shaped space. Notice that some Banach spaces are not Busemann
spaces~\cite{Papa05}.
\end{example}
The next example relies on the following notion.
\begin{defi}\label{def-heminorm}
A map $p$ from a Banach space $(X,\|\cdot\|)$ to $\mathbb{R}$ is a  {\em hemi-norm} if
it can be written as
\begin{align}
p(z)=\sup_{\phi \in E} \phi(z)
\label{e-hemi-norm}
\end{align}
where $E$ is a bounded subset of the dual space $X^{\star}$, and
\begin{align}\label{cond-hemi}
p(z)=p(-z)=0\implies z= 0 \enspace .
\end{align}
\end{defi}
A hemi-norm is always Lipschitz.
It is easily verified that if $p$ is a hemi-norm, %
then $\delta(x,y):= p(y-x)$ is a hemi-metric. Note that Condition~\eqref{cond-hemi} holds if and only if the orthogonal set of $E$ is reduced to the zero vector. As in the case of Banach spaces, $(X,\delta)$ is metrically star shaped: it suffices to choose
the straight lines as geodesics. We shall say that
the hemi-norm is {\em compatible} with the Banach space $(X,\|\cdot\|)$
if $\|z\|=\max(p(z),p(-z))$.  Then, the metric $d$ obtained
from $\delta$ is the one associated to the norm of $X$, and
in particular, the hemi-metric space $(X,\delta)$ is complete.

\begin{example}
The norm of the Banach space $X$ is a special case of compatible hemi-norm,
in which
 \[p(z)=\|z\|=\sup_{\phi \in B^{\star}} \psi(z)\]
where $ B^{\star}$ is the dual unit ball of $X^{\star}$.
Then, by Bauer maximum principle~\cite[Theorem~7.69]{aliprantisborder}, the supremum is attained at some extreme point of $ B^{\star}$, so the set $E$
arising in the definition of the hemi-norm
(see Equation~\ref{e-hemi-norm})
can be taken to be either $B^{\star}$ or the set of extreme points of $B^{\star}$.
\end{example}
\begin{example}\label{exampletop}
A useful example of hemi-norm is the following.
Let $\Omega$ denote a compact topological space, let $X=\mathscr{C}(\Omega)$
denote the Banach space of continuous functions from $\Omega$ to $\RR$,
equipped with the sup-norm. Consider $\delta(x,y)=t(y-x)$, where $t$ denotes the ``top'' operator which gives the maximum of a function,
i.e.,
\[
t(x)=\max_{\omega\in\Omega}x(\omega) \enspace.
\]
The metric $d$ obtained from $\delta$ is the sup-norm.
We can take for $E$ the set of evaluations functions $\{e_\omega, \omega\in \Omega\}$ with $e_\omega(x)=x(\omega)$.
\end{example}

\begin{example}
Another class of examples concern symmetric cones,
which include cones of positive semidefinite matrices.
We shall discuss them further in Section~\ref{sec-cones}.
\end{example}

\subsection{Non-expansive mappings}

In this section we consider a metrically star-shaped space $(X,\delta)$ with centre $\centerpoint{x}$, as well as a map $T:X\rightarrow X$ that is non-expansive with respect to the hemi-metric $\delta$, meaning that for all $(x,y)\in X^2$,
\[
\delta(T(x),T(y))\leq \delta(x,y).
\]

\begin{defi}
To each non-expansive mapping $T$, we associate the two following quantities:
\begin{equation}
\overline{\rho}(T)=\inf_{x\in X} \delta\left(x,T(x)\right) \enspace,
\label{defrhobar}
\end{equation}
\begin{equation}\label{deftimecycle}
\rho(T)=\lim_{k\rightarrow+\infty}\frac{\delta\left(x,T^k(x)\right)}{k}=\inf_{k\geq1}\frac{\delta\left(x,T^k(x)\right)}{k} \enspace.
\end{equation}
\end{defi}
Thus, the number $\rho(T)$ measures the {\em linear escape rate} of the orbits of $T$. It is well defined. Indeed,
 for any $x\in X$,
since $T$ is non-expansive,
the sequence $u_k=\delta\left(x,T^k(x)\right)$ is {\em subadditive}, meaning that $u_{k+l}\leq u_k+u_l$ for all $k,l\geq 1$, and so, a classical argument (see~\cite[Lemma~1.18]{bowen}) shows that the limit $\lim_{k\rightarrow+\infty}k^{-1}\delta\left(x,T^k(x)\right)$ does exist and is equal to $\inf_{k\geq1}k^{-1}\delta\left(x,T^k(x)\right)$.
Moreover, since
    \begin{eqnarray*}
    \delta\left(x,T^k(x)\right)&\leq& \delta(x,y)+\delta\left(y,T^k(y)\right)+\delta\left(T^k(y),T^k(x)\right)\\
    &\leq& \delta(x,y)+\delta(y,x)+\delta\left(y,T^k(y)\right)
    \end{eqnarray*}
the limit is independent of the choice of $x\in X$.

 \begin{lem}\label{inegalitefacile}
 The following inequality is satisfied for any non-expansive mapping $T$:
 \[\rho(T)\leq\overline{\rho}(T).\]
 \end{lem}

 \begin{proof}
Observe that for any $x\in X$ and $k\geq1$,
\begin{align*}
\delta\left(x,T^k(x)\right)&\leq\sum_{l=0}^{k-1}\delta\left(T^l(x),T^{l+1}(x)\right)
\leq
k\delta(x,T(x))
\end{align*}
and take the infimum on both $k$ and $x$.
 \end{proof}

 \subsection{The horofunction boundary}\label{subsec-horo}
The horofunction boundary of a metric space was defined by Gromov~\cite{gromov78}. See also~\cite{gromov}, \cite[Ch.~II]{ballmann} and~\cite{rieffel}.
In fact, the same construction can be performed with a hemi-metric.
The details can be found in~\cite{AGW}, in which
a horofunction-like boundary is defined for discrete optimal control
problems, by analogy with the Martin compactification arising
in the theory of Markov processes. See also~\cite{ishiimitake}.

Let us fix an arbitrary point $\basepoint{x}\in X$ (the {\em basepoint}).
We define a map $i$ from $X$ to the set functions from $X$ to $\mathbb{R}$ by associating to any $x\in X$ the following function $i(x)$:
\[
i(x):y\rightarrow [i(x)](y):=\delta(\basepoint{x},x)-\delta(y,x).
\]
Using the triangular inequality, we get that, for all $x,y\in X$,
\[
-\delta(y,\basepoint{x})\leq [i(x)](y)\leq \delta(\basepoint{x},y)
\]
Hence, the set $i(X):=\{i(x)\mid x\in X\}$ can be identified to a
subset of the product space
 $\prod_{y\in X}[-\delta(y,\basepoint{x}),\delta(\basepoint{x},y)]\subset\RR^X$. By Tychonoff's theorem, the later space
is compact for the product topology.
Hence, the closure of $i(X)$ in the topology
of pointwise convergence (which is the same as the product
topology) is compact.
We denote by $\sM$ this closure. %
The elements of the {\em boundary} $\sH:=\sM\setminus i(X)$
are called {\em horofunctions}. We will call {\em Martin} functions
the elements of $\sM$ (the {\em Martin space}). Note that the choice of the basepoint
is irrelevant (changing the basepoint translates
all the Martin functions by the same constant).

\begin{remark}
The term of horofunction is sometimes used in a more restrictive sense, to denote what is also called a {\em Busemann function}, which is a horofunction obtained as the limit of a family of
functions $(i(x_s))_{s\geq 0}$ taken along an "infinite geodesic"
$(x_s)_{s\geq 0}$, see~\cite{rieffel}. Note also the sign in the definition of $i(x)$ (which implies, in the Banach space case, that the map $i(x)$ is concave).
This sign is chosen consistently with potential theory (the opposite of the distance is the analogue of the Martin kernel). In the metric geometry literature, the opposite choice is most often made.
\end{remark}
\begin{remark}\label{rk-technical}
The horofunctions are often defined as the closure of $i(X)$ in the topology
of uniform convergence on bounded sets;
indeed, the injection $x\mapsto i(x)$ is continuous for this topology,
and this injection is an embedding if $X$ is a complete
geodesic space~\cite[Ch.~II]{ballmann}. It can be checked that every
map $i(x)$ is Lipschitz of constant $1$ with respect to the metric $d$, and so,
by the Ascoli-Arzela theorem, the closure of $i(X)$ in this sense is compact as soon as $X$ is proper (meaning that every closed ball is compact).
In the present work, we do not require $X$ to be proper,
but define rather $\mathscr{M}$ as the closure of $i(X)$ in the topology of
pointwise convergence, so that $\mathscr{M}$ is always compact.
Note however that the injection $x\mapsto i(x)$ may {\em not}
be an embedding from $X$ to $\mathscr{M}$ (the topology on $\mathscr{M}$
is too weak in general, so the inverse of the map $x\mapsto i(x)$ may not be continuous).
\end{remark}
\begin{remark}\label{remark-topo-martin}
The topology of the Martin space is metrisable as soon as $X$ is a countable union
of compact sets, see for instance~\cite[Remark~7.10]{AGW}.
Then, every horofunction
is the limit of a sequence of functions $i(x_m)$, where $x_m$ is a sequence of elements of $X$.
\end{remark}
\section{The main result and some of its consequences}
\subsection{The main result}
We shall derive Theorem~\ref{th-minimax} from the following result.
\begin{theorem}\label{maintheorem}
Let $(X,\delta)$ be a complete metrically star-shaped hemi-metric space, and let $T:X\rightarrow X$ be non-expansive. Then there exists a Martin function $h\in\sM$ such that for all $x\in X$,
\[
h(T(x))\geq h(x)+\overline{\rho}(T) \enspace.
\]
Moreover, if $\delta$ is bounded from below (in particular, if $\delta$ is a metric) and if $\overline{\rho}(T)>0$, then $h$ is necessarily a horofunction.
\end{theorem}

\begin{proof}
Let $\centerpoint{x}$ and $\{\gamma_y\}_{y\in X}$ as in Definition \ref{defistarshaped}; for any $\alpha\in[0,1[$ denote by $r^\alpha:X\rightarrow X$ the function $y\rightarrow \gamma_y(\alpha)$ and recall that $r^\alpha$ is $\alpha$-contracting by definition. By completeness, we can thus define for any $\alpha$ the point $y_\alpha\in X$ as the only solution of the fixed point equation
\[
T(r^\alpha(y_\alpha))=y_\alpha.
\]
Then for any $x\in X$,
\begin{eqnarray*}
\delta(x,y_\alpha)-\delta\left(T(x),y_\alpha\right)&=&\delta(x,y_\alpha)-\delta\left(T(x),T(r^\alpha(y_\alpha))\right)\\
&\geq&\delta(x,y_\alpha)-\delta\left(x,r^\alpha(y_\alpha)\right)\quad \text{(by non-expansiveness)}\\
&\geq&\delta(x,y_\alpha)-\delta\left(x,r^\alpha(x)\right)-\delta\left(r^\alpha(x),r^\alpha(y_\alpha)\right)\\
&\geq&(1-\alpha)\delta(x,y_\alpha)-\delta\left(x,r^\alpha(x)\right)\quad \text{(since $r^\alpha$ is an $\alpha$-contraction)}\\
&\geq&(1-\alpha)\delta(\centerpoint{x},y_\alpha)-(1-\alpha)\delta(\centerpoint{x},x)-\delta\left(x,r^\alpha(x)\right)\\
&=&\delta(r^\alpha(y_\alpha),y_\alpha)-(1-\alpha)\delta(\centerpoint{x},x)-\delta\left(x,r^\alpha(x)\right)\\
&\geq&\delta\left(T(r^\alpha(y_\alpha),T(y_\alpha))\right)-(1-\alpha)\delta(\centerpoint{x},x)-\delta\left(x,r^\alpha(x)\right)\\
&&\qquad\qquad \qquad\qquad\qquad\qquad\text{(by non-expansiveness)}\\
&=&\delta\left(y_\alpha,T(y_\alpha))\right)-(1-\alpha)\delta(\centerpoint{x},x)-\delta\left(x,r^\alpha(x)\right)\\
&\geq&\overline{\rho}(T)-(1-\alpha)\delta(\centerpoint{x},x)-\delta\left(x,r^\alpha(x)\right).
\end{eqnarray*}
Since the space $\sM$ is compact, the family
of functions $(i(y_\alpha))_{0<\alpha<1}$ admits a limit point $h\in\sM$ as $\alpha$ tends to $1$. Passing to the limit in the previous
inequality, and using the additional assumption in Definition \ref{defistarshaped}, we deduce that $-h(x)+h(T(x))\geq \overline{\rho}(T)$.

Assume now by contradiction that $\delta$ is bounded from below, that $\overline{\rho}(T)>0$,
and that $h=i(z)$ for some $z\in X$. Then, we deduce from $h(T(x))\geq
h(x)+\overline{\rho}(T)$ and $h(y)=-\delta(y,z)+\delta(b,z)$ that
\[
-\delta(T^k(x),z)+\delta(x,z)\geq k\overline{\rho}(T) \enspace .
\]
The right-hand side of this expression tends to $+\infty$ as $k$ tends to infinity, but the left-hand side of this expression is bounded above independently of $k$, since $\delta(T^k(x),z)$ is bounded from below. This is impossible, hence, $h\in \sH$
is a horofunction.
(More generally, the same argument shows that for any net of functions
$(i(x_\beta))_{\beta \in B}$ converging to $h$, the net $(x_\beta)_{\beta\in B}$
cannot have a bounded subnet.)
\end{proof}

We obtain as immediate consequences the following corollaries:
\begin{corollary}\label{cor-pumping}
Let $X,\delta$  and $T$ be as in Theorem~\ref{maintheorem}. Then,
there exists a Martin function $h\in\sM$ such that for all $x\in X$,
\begin{align}\label{eq-pumping}
h(T^k(x))\geq h(x)+k \overline{\rho}(T) \enspace .
\end{align}
Moreover, if $\delta$ is bounded from below and if $\overline{\rho}(T)>0$, then $h$ is necessarily a horofunction.
\end{corollary}
\begin{corollary}\label{coregalite}
If the assumptions of Theorem \ref{maintheorem} are satisfied, then $\overline{\rho}(T)=\rho(T)$.
\end{corollary}

\begin{proof}
We deduce from~\eqref{eq-pumping} that
\begin{eqnarray*}
\overline{\rho}(T)&\leq& \frac{1}{k} \left(h(T^k(x))- h(x)\right)
\leq\frac{\delta(x,T^k(x))}{k}
\end{eqnarray*}
which yields $\overline{\rho}(T)\leq \rho(T)$. Then, the result follows by Lemma \ref{inegalitefacile}.
\end{proof}
Theorem~\ref{th-minimax} follows readily by combining Theorem~\ref{maintheorem} and Corollary~\ref{coregalite}.
\subsection{The Kohlberg-Neyman theorem revisited}
We now apply Theorem~\ref{maintheorem} to non-expansive self-maps of a Banach space $X$, and more generally, to those maps that are non-expansive in a {\em compatible hemi-norm}, as defined in Section~\ref{subsec-hemi}.
 \begin{theorem}\label{seminormtheorem}
 Let $(X,\|\cdot\|)$ be a Banach space, and let
 \[p(z)=\sup_{\phi \in E} \phi(z)\]
be a compatible hemi-norm, where $E$ is a bounded subset of
the dual space $X^{\star}$.
Let $T$ be non-expansive for $\delta$. Then, for every $x\in X$, there exists a linear form $\phi$ in the weak-star closure of $E$ such that for every $k\in \mathbb{N}$,
 \[\phi(T^k(x))\geq \phi(x)+k \overline{\rho}(T).
 \]
 \end{theorem}

 \begin{proof}
Let $\closure{E}$ denote the closure of $E$ in the weak-star topology.
Observe that, for all $u\in X$,
\[
p(u)=\sup_{\phi\in \closure{E}}\phi(u) \enspace.
\]
Since $E$ is bounded, $\closure{E}$ is weak-star compact,
and so, the map $\phi\mapsto \phi(u)$
attains its maximum on $\closure{E}$.

Fix now $x$ and $y$ in $X$. It follows from the previous discussion that
for all $z\in X$,
\[
i(y)(z)-i(y)(x)=p(y-x)-p(y-z)=\phi(y-x)-p(y-z)
\]
for some $\phi \in \closure{E}$,  independent of $z$, and so
\[
i(y)(z)-i(y)(x)\leq\phi(y-x)-\phi(y-z)=\phi(z-x).
\]
In other words, $\phi$ is in the super-differential of $i(y)$ at point $x$.

Let $h$ be the Martin function that appears in Corollary \ref{cor-pumping}; $h=\lim_D i(y_d)$ for some net $(y_d)_{d\in D}$. By the previous observation, for every $d\in D$ we can find an element $\phi_d\in \closure{E}$ such that for every $k\in \mathbb{N}$,
\[
i(y_d)(T^k(x))-i(y_d)(x)\leq \phi_d(T^k(x)-x).
\]
Taking the limit along some subnet, we can find a linear form $\phi$ in
$\closure{E}$ such that for every $k\in \mathbb{N}$,
\[
h(T^k(x))-h(x)\leq \phi(T^k(x)-x).
\]
Since by Corollary \ref{cor-pumping} $h(T^k(x))-h(x)\geq k \overline{\rho}(T)$, the theorem is established.
 \end{proof}

We now recover the following result.

 \begin{corollary}[Kohlberg and Neyman~{\cite{KoNe,neymansurv}}]\label{linear}
Let $(X,\|\cdot\|)$ be a Banach space, let $T:X\rightarrow X$ be non-expansive
and assume that $\overline{\rho}(T)>0$.
Then for any $x\in X$, there exists a continuous linear form $\phi$ of norm one, such that
\begin{align}
\phi(T^k(x))\geq \phi(x)+k\overline{\rho}(T)\label{e-kn}
\end{align}
for all $k\in\mathbb{N}$. Moreover, $\phi$ can be taken in the weak-star closure of the set of extreme points of the dual unit ball.
\end{corollary}
\begin{proof}
We apply Theorem \ref{seminormtheorem}, taking $E$ to be the set of extreme points of the dual unit ball. This gives a $\phi$  in the weak-star closure of $E$ satisfying (\ref{e-kn}). In particular $\|\phi\|_\star\leq 1$. Moreover,
$\overline{\rho}(T)\leq \lim_{k\to\infty}k^{-1}\phi(T^k(x)-x)
\leq \|\phi\|_{\star} \rho(T)$,
and since $\rho(T)=\overline{\rho}(T)>0$, $1\leq \|\phi\|_{\star}$,
which shows that $\|\phi\|_{\star}=1$.
\end{proof}

Actually, the assumption that $\overline{\rho}(T)>0$ can be dispensed with in the finite dimension case:

 \begin{corollary}\label{linearfinite}
Let $(X,\|\cdot\|)$ be a finite dimensional Banach space, and let $T:X\rightarrow X$ be non-expansive.
Then for any $x\in X$, there exists a linear form $\phi$ of norm one, in the
closure of the set of extreme points of the unit ball, such that
\begin{align}
\phi(T^k(x))\geq \phi(x)+k\overline{\rho}(T)
\end{align}
for all $k\in\mathbb{N}$.
\end{corollary}

\begin{proof}
The proof is similar to the one of Corollary \ref{linear}. We still take $E$ as the set of extreme points of the dual unit ball, which belongs to the dual unit sphere, and we conclude since the unit sphere of a finite dimensional Banach space is closed.
\end{proof}

\begin{remark}
The case in which
$\rho(T)=0$,  in finite dimension,
has been studied in particular by Lins in~\cite{lins2009}, who proved
that either $T$ has a fixed point or there exists a linear form $\varphi$ of norm $1$ such that $\varphi(T^k(x))\to +\infty$ for all $x\in X$. This result is established by showing first that in the latter case,
there is a horofunction $h$ such that $h(T^k(x))\to+\infty$ as $k$ tends to $\infty$ (modulo a change of sign convention for horofunctions).
 As noted in Remark~3.1 of~\cite{lins2009}, the construction
of the linear form $\varphi$
from the horofunction $h$ relies essentially on a sub-differentiability argument.
An argument of the same nature is used here to derive Corollary~\ref{linearfinite}
from Theorem~\ref{maintheorem}. Note however that
Corollary \ref{linearfinite}
and the existence of the horofunction in~\cite{lins2009} are independent
results (none of them can be recovered from the other one).
\end{remark}

\subsection{Application to Shapley operators}\label{appli-shapley}

In this section $\Omega$ is a compact topological space, and $X=\mathscr{C}(\Omega)$ is the Banach space of continuous functions from $\Omega$ to $\mathbb{R}$, endowed with the sup norm.
We will be interested in functions $T:X\rightarrow X$ that are {\em order-preserving}, meaning that
\[ x\leq y \implies T(x)\leq T(y)
\]
where $\leq$ denotes the canonical partial order on functions.
We shall say that $T$ {\em commutes with the addition of a constant}
if
\[
T(\lambda + x) =\lambda +T(x), \qquad \forall \lambda\in \RR \enspace,
\]
where $\lambda+x$ denotes the function $\omega\mapsto \lambda+x(\omega)$.

As in Example~\ref{exampletop}, we set $t(x):=\max_{\omega\in \Omega} x(\omega)$ and consider
$\delta(x,y)=t(y-x)$. Recall that $\delta$ is a hemi-metric compatible with the sup norm and that $(X,\delta)$ is metrically star-shaped.

It is easy to see that $T$ is order-preserving and commutes with the addition of a constant if and only if
\[
t(T(x)-T(y))\leq t(x-y)
\]
for all $x$ and $y$, meaning that $T$ is non-expansive for $\delta$
(see in particular~\cite{GuKe}, or~\cite{lemmens}).

For this special hemi-metric $\delta$, we denote by $\overline{\rho}_+(T)$ and $\rho_+(T)$ the quantities defined in (\ref{defrhobar}) and (\ref{deftimecycle}).
Theorem~\ref{maintheorem} and Corollary \ref{coregalite} imply that

\[
\rho_+(T)=\lim_{k\to +\infty}\max_{\omega\in \Omega}\frac {T^k(x)(\omega)-x(\omega)}{k}=\overline{\rho}_+(T)=\inf_{y\in X} \max_{\omega\in \Omega} (T(y)(\omega)-y(\omega))
\]
for all $x$ in $X$.

We next derive from Theorem \ref{seminormtheorem} the following
generalisation of a result established in \cite{GG04} when $\Omega$ is finite.

 \begin{corollary}\label{linearcompact}
Let $\Omega$ be a compact topological space, and let $X=\mathscr{C}(\Omega)$ be the Banach space of continuous functions from $\Omega$ to $\mathbb{R}$, endowed with the sup norm. Suppose that $T:X\rightarrow X$ is order-preserving and commutes with the addition of a constant.
Then for every $x\in X$, there exists $\omega_+\in \Omega$ such that,
\begin{align}
T^k(x)(\omega_+)\geq x(\omega_+)+k\overline{\rho}_+(T)
\end{align}
for all $k\in\mathbb{N}$.
\end{corollary}

\begin{proof}
We apply Theorem~\ref{seminormtheorem} to the hemi-norm
$p(z)=t(z)$, taking for $E$ the set of evaluation functions $e_\omega$,
$\omega\in \Omega$.
Since $\Omega$ is compact, the set $E$ is weak-star closed,
so that the linear form $\phi$ in Theorem~\ref{seminormtheorem} can be
written as $\phi=e_{\omega_+}$ for some $\omega_+\in E$.
\end{proof}

Order-preserving maps $T$ commuting with the addition of a constant arise as {\em Shapley operators} in game theory. Suppose now that $\Omega$, $A$ and $B$ are three compact metric spaces, that $g$ is a continuous function from $A\times B\times \Omega$ to $\mathbb{R}$, and that $q:A\times B\times \Omega\rightarrow \Delta(\Omega)$ is continuous, where $\Delta(\Omega)$, the set of probabilities on $\Omega$, is endowed with the weak topology. Then the Shapley operator of the two-player, zero-sum stochastic game with state space $\Omega$, action sets $A$ and $B$, payoff $g$, and transition probability $q$, is defined as

\begin{align*}
T(x)(\omega)&=\sup_{\mu\in \Delta(A)} \inf_{\nu\in \Delta(B)} \left\{ \int_{A\times B} \left[g(a,b,\omega)+\int_\Omega q(dw'|a,b,\omega) x(\omega')\right ] \mu(da)\nu(db) \right\}\\
&=\inf_{\nu\in \Delta(B)} \sup_{\mu\in \Delta(A)} \left\{ \int_{A\times B} \left[g(a,b,\omega)+\int_\Omega q(dw'|a,b,\omega) x(\omega')\right ] \mu(da)\nu(db) \right\} \enspace.
\end{align*}
The map $T$ is order-preserving and commutes with the addition of a constant.
It is shown in~\cite{MSZ,Nowak} that it preserves
$\mathscr{C}({\Omega})$. %
Then, $\frac{T^k(0)(\omega)}{k}$ represents the average reward per time unit if both players play optimally in the $k$-stage game and if the starting state is $\omega$. Thus, Corollary \ref{linearcompact} shows that there is an initial state $\omega_+$ which is the best for the maximising player in the long term. In particular $\frac{T^k(0)(\omega_+)}{k}$ converges as $k$ goes to infinity.

\begin{remark}
To every result concerning the hemi-norm $t$ corresponds
a dual result concerning the hemi-norm $b(x):=-t(-x)=\min_{\omega\in \Omega} x(\omega)$.
In particular,
\[
\rho_-(T):=\lim_{k\to +\infty}\min_{\omega\in \Omega}\frac {T^k(x)(\omega)-x(\omega)}{k}=\overline{\rho}_-(T):=\sup_{y\in X} \min_{\omega\in \Omega} (T(y)(\omega)-y(\omega))
\]
Moreover, there exists $\omega_-\in \Omega$ satisfying
\begin{align*}
T^k(x)(\omega_-)\leq x(\omega_-)+k\overline{\rho}_-(T)
\end{align*}
for all $k\in\mathbb{N}$. This is readily deduced by applying Corollary \ref{linearcompact} to the map $x\rightarrow -T(-x)$.
\end{remark}

\subsection{Comparison with a theorem of Karlsson}
Theorem~\ref{maintheorem} should be compared with the following result.
(Recall that a metric space is {\em proper} if
every closed ball is compact.)
\begin{theorem}[Karlsson~{\cite[Th.~3.3]{Ka}}]\label{thmKarlsson}
If $T$ is a non-expansive self-map of a proper metric space $(X,\delta)$,
then, for all $x\in X$, there exists a horofunction $h$ (depending
on $x$), such that
\[
h(T^k(x))\geq h(x)+k{\rho}(T) \enspace,
\]
 holds for all $k \geq 1$.
\end{theorem}
Note first two differences in the statements of Theorem~\ref{maintheorem} and
of Karlsson's theorem: we require a metric convexity assumption whereas
Karlsson does not; Karlsson makes the general assumption that the metric
space is proper whereas we do not.
Actually, the compactness issue appears to be a mere technicality:
it can be checked that the proof of Theorem~\ref{thmKarlsson}
remains valid even if $X$ is not proper if one defines horofunctions with respect to the topology
of pointwise convergence as we do here (see Remark~\ref{rk-technical}).

A key difference however lies in the conclusions of both results. In Theorem~\ref{maintheorem}, horofunctions are {\em independent} of the choice of $x$, whereas
in Karlsson's result, the horofunction does depend on $x$. Note also
that inequalities involving $\overline{\rho}(T)$ as in Theorem~\ref{maintheorem}
and Corollary~\ref{cor-pumping} are stronger than their counterparts involving
$\rho(T)$ since $\overline{\rho}(T)\geq \rho(T)$.
This is illustrated in Example~\ref{examplenonconvexemetrique} below,
in which we consider
a non-star-shaped space for which Theorem~\ref{thmKarlsson} applies but not Theorem~\ref{maintheorem}. This example shows that the assumptions of Theorem~\ref{maintheorem} cannot be weakened : it is primordial that the
space $X$ is metrically star-shaped.
Hence, Theorem~\ref{maintheorem} and Karlsson's result are
incomparable: the latter holds under more general circumstances,
but the former yields a stronger conclusion.

\begin{example}\label{examplenonconvexemetrique}
Let $\T:=\RR (\text{mod} 1)$ be the torus with its canonical metric
$d_\T$; and let $X=\T\times \RR$, with the metric
\[
d((x,t),(x',t')):=d_\T(x,x')+|t-t'| \enspace .
\]
Let $\alpha\in(0,\frac{1}{2})\setminus \Q$, and
consider the mapping $T:X\to X$,
\[
T(x,t)=(x+\alpha,t+1)
\]
which is non-expansive for the metric $d$. It is straightforward to
check that $\rho(T)=1$ but $\overline{\rho}(T)=1+\alpha$, so
inequation~\eqref{eq-pumping} cannot be satisfied (otherwise one would
have $\rho(T)=\overline{\rho}(T)$ by Corollary \ref{coregalite}).

In fact we prove that it is not even true that there exists a
function $h\in\sM$ such that for all $x$,
\begin{align}
h(T^k(x,t))\geq h(x,t)+k\rho(T),\qquad \forall k\in \NN.
\end{align}

Fix a basepoint $(\basepoint{x},\basepoint{t})$ in $X$. Since the metric on
$X=\T\times \RR$ is the sum of the two metrics $(\T,d_T)$ and
$(\RR,|\cdot|)$, we can write any Martin function $h$ of $X$ as the sum
\[
h(x,t)=h_1(x)+ h_2(t)
\]
where $h_1\in $ is a Martin function of $(\T,d_T)$ and $h_2$ is a
Martin function of $(\RR,|\cdot|)$, relatively to the basepoints
$\basepoint{x}$ and $\basepoint{t}$, respectively (see Prop.~9.1 in \cite{AGW}).
The only Martin functions on $\T$ are
$h_1(x)=-d_{\T}(x,y)+d_{\T}(\basepoint{x},y)$ with $y\in \T$. The
Martin functions of $(\RR,|\cdot|)$ are either of the form
$h_2(t)=-|t-s|+|\basepoint{t}-s|$, with $s\in \RR$, or the affine
functions $h_2(t)=t-\basepoint{t}$ and $h_2=-(t-\basepoint{t})$. We search a
Martin function $h$ such that, for any $(x,t)\in X$,
\begin{align}
h(T^k(x,t))\geq h(x,t)+k,\qquad \forall k\in \NN \label{e-h}
\end{align}
i.e.,
\begin{align}
h_1(x+k\alpha)+h_2(t+k)\geq  h_1(x)+h_2(t)+k, \qquad \forall k\in
\NN. \label{ineq-pump}
\end{align}
Any Martin function $h_1$ is bounded. If $h_2$ is of
the form $h_2(t)=-|t-s|+|\bar{t}-s|$ or 
$h_2(t)=-(t-\bar{t})$, then, for a fixed $(x,t)$,
the left-hand side
of~(3.7) remains bounded from above as $k\to\infty$, 
whereas the right-hand side of~(3.7) tends to infinity. The only possibility
is thus $h_2(t)=t-\bar{t}$.
In that case,~\eqref{ineq-pump} becomes
\[
h_1(x+k\alpha)-h_1(x)\geq 0.
\]
Since the sequence $(x+k\alpha)_{k\in\NN}$ is dense in $\T$, and
since $h_1$ is continuous, we have that $h_1(x)\leq \inf_{y\in
\T}h_1(y)$. This is true for any $x\in \T$, so $h_1$ is constant.
But there are no constant Martin function of $(\T,d_\T)$.

Karlsson's result shows however that for any choice of $(x,t)\in X$,
there exists a horofunction $h$ (depending on $(x,t)$) such
that~\eqref{e-h} is satisfied.
\end{example}

\section{A Denjoy-Wolff type theorem for order-preserving homogeneous self-maps of symmetric cones}
\label{sec-cones}
\subsection{The reverse Funk metric}\label{sec-funk}
In this section, we recall or establish preliminary results concerning cones.
We consider a (closed, convex and pointed) cone $C$ in $\RR^n$
with non-empty interior $X:=\Int C$.
We equip $C$ with the partial order defined by $x\leq y$ if $y-x\in C$. For all $x,y\in C\setminus\{0\}$, we set
\[
M(y/x):= \inf\{\lambda>0\mid \lambda x\geq y\} , \qquad
m(y/x):= \sup\{\lambda>0\mid \lambda x\leq y\}
\enspace .
\]
Since $C$ is closed, as soon as the two latter sets are non-empty,
their respective infimum and supremum are attained;
in particular, $M(y/x)$ and $m(y/x)$ are finite and positive.

The (reverse) {\em Funk hemi-metric} on the interior of $C$
is defined by
\[\Funk(x,y):=\log M(y/x) \enspace.
\]
More generally, we shall use the notation $\Funk(x,y)$ as soon
as $x,y\in C\setminus\{0\}$ are such that $\mu x\geq y$
for some $\mu>0$. The terminology ``Funk metric''
is borrowed from~\cite{papadopoulos,walsh}; it refers
to~\cite{funk}.
The map $\delta:=\Funk$ is easily checked to be a hemi-metric on $\Int C$.
Indeed:
$\delta(x,y)$ is finite since $x$ is in the interior
of $C$;
the triangular inequality and the fact that $\delta(x,x)=0$ for any $x\in X$
are easy to verify since $C$ is pointed;
the fact that $\delta(y,x)=\delta(y,x)=0\Longrightarrow x=y$
follows from the fact that $C$ is closed and pointed.
The term ``reverse'' is by opposition to the hemi-metric
$(x,y)\mapsto \delta(y,x)$ which has a different horoboundary~\cite{walsh}.

The quantities $M(y/x)$ and $m(y/x)$ can be expressed in terms
of the extreme rays of the {\em dual cone} $C^\star$. The latter
is the set of continuous linear forms taking non-negative values on $C$.
For each cone $K$, $\Extr K$ denotes
the set of {\em representatives} of the extreme rays of $K$
(i.e., the non-zero vectors belonging to these extreme rays).
\begin{lem}\label{lem-funk}
Let $x,y,z\in C\setminus\{0\}$ be such that $\mu x\geq y$ and $z\geq \nu x$
for some $\mu,\nu>0$. Then,
\begin{align}
M(y/x)&= \sup_{\varphi\in C^\star\!,\; \varphi(x)\neq 0} \frac{\varphi(y)}{\varphi(x)}
= \sup_{\varphi\in \Extr C^\star\!,\;  \varphi(x)\neq 0} \frac{\varphi(y)}{\varphi(x)}\enspace,\label{e-def-funk}\\
m(z/x)&= \inf_{\varphi\in C^\star\!,\;  \varphi(x)\neq 0} \frac{\varphi(z)}{\varphi(x)}
= \inf_{\varphi\in \Extr C^\star\!,\; \varphi(x)\neq 0} \frac{\varphi(z)}{\varphi(x)}\enspace.\label{e-def-funkdual}
\end{align}
Moreover, when $x\in \Int C$, the latter suprema are attained,
and the condition that $\varphi(x)\neq 0$ can be replaced by $\varphi\in C^\star\setminus\{0\}$. When $z\in \Int C$, the latter infima are attained.
\end{lem}
This result is somehow standard. We include the proof for the convenience of the reader.
\begin{proof}
Since $C=(C^\star)^\star$, the inequality $\lambda x\geq y$, i.e., $\lambda x-y\in C$, is equivalent to $\varphi(\lambda x -y)\geq 0$, for all $\varphi\in C^\star$, from which the first equality in~\eqref{e-def-funk} follows.

Let now recall that if $u$ is in the interior of $C$, and if $\varphi\in C^\star\setminus \{0\}$, then, $\varphi(u)$ cannot vanish. Indeed, we can find
a ball $B$ centred at $0$, such that $u+ z\in C$ for all $z\in B$.
If $\varphi(u)=0$, then, $0\leq \varphi(u\pm z)= \pm\varphi(z)$ holds
for all $z\in B$, and so $\varphi=0$.

Choose now a vector $u\in \Int C$,
and define $\Sigma_u:= \{\varphi\in C^\star\mid \varphi(u)=1\}$.
Observe that $\Sigma_u$ is compact and convex, and that
each extreme ray of $C$ can be identified to the unique
extreme point of $\Sigma_u$ which generates this ray.
Let $\bar{\lambda}$ denote the value of the last supremum in~\eqref{e-def-funk},
so that the inequality $\bar{\lambda} \varphi(x)\geq \varphi(y)$ holds for all extreme points $\varphi$ of $\Sigma_u$. Since every element $\varphi$
of $\Sigma_u$ is a barycentre of extreme points of $\Sigma_u$,
we deduce that the same inequality holds for all points $\varphi$ of $\Sigma_u$, which shows that $\bar{\lambda}$
coincides with the first supremum in~\eqref{e-def-funk}.

The arguments to establish~\eqref{e-def-funkdual} are dual.

Finally, when $x\in \Int C$, we consider the map
$J:\Sigma_x\to \RR$, $\varphi\mapsto \varphi(y)/\varphi(x)=\varphi(y)$.
Since $J$ is linear, it attains its maximum at an extreme
point of the compact convex set $\Sigma_x$, which
implies that both suprema in~\eqref{e-def-funk} are
attained. Similarly,  when $z\in \Int C$, we note that
every $\varphi\in \Sigma_z$ maximising $\varphi(x)$
attains the first infimum in~\eqref{e-def-funkdual},
and so, the second infimum is attained
at an extreme point of $\Sigma_z$.
\end{proof}

Walsh showed in~\cite{walsh} that the horoboundary of $\Int C$ in the (reverse) Funk hemi-metric coincides with the space of rays contained in the Euclidean boundary of $C$.
\begin{proposition}[{\cite[Prop.~2.5]{walsh}}]
Let $C$ be a closed convex pointed cone of non-empty interior in a finite dimensional Banach space, with basepoint $\basepoint{x}$.
Then, any Martin function $h$ for the (reverse) $\Funk$ hemi-metric corresponds to a vector $u\in C\setminus\{0\}$,
\begin{align}\label{e-martinfunction}
h(x)=-\Funk(x,u)+\Funk(\basepoint{x},u) \enspace,\forall x\in \Int C,
\end{align}
and $h$ is a horofunction if and only if $u\in \partial C\setminus\{0\}$.
\end{proposition}
We shall also be interested in the following variant of the Funk hemi-metric,
considered by Papadopoulos and Troyanov~\cite{papadopoulos}:
\[
\Funk^+(x,y):=\max(\Funk(x,y),0) \enspace.
\]
The map $\Funk^+$ is easily seen to be a hemi-metric.
The following simple application
of the arguments of~\cite{walsh} characterises
the boundary of the hemi-metric $\Funk^+$.
\begin{proposition}\label{prop-walsh-funkp}
Let $C$ be a closed convex pointed cone of non-empty interior in a finite dimensional Banach space.
Then, the Martin space for the $\Funk^+$ hemi-metric consists of the functions~\eqref{e-martinfunction}, with $u\in C\setminus\{0\}$, together with the functions
\begin{align}
h(x)=-\Funk^+(x,u)+\Funk^+(\basepoint{x},u) \enspace,\forall x\in \Int C,
\label{e-martinfunctionsfunkplus}
\end{align}
with $u\in C$, and an element is a horofunction if and only if it is of
the form~\eqref{e-martinfunction}, with $u\in C\setminus\{0\}$, or of the form~\eqref{e-martinfunctionsfunkplus} with $u\in \partial{C}$.
\end{proposition}
\begin{proof}
Any function $h$ of the Martin space
is the pointwise limit of a sequence of functions $h_m: x\mapsto -\Funk^+(x,z_m)+\Funk^+(\basepoint{x},z_m)$, where $(z_m)_{m\geq 1}$ is a sequence of elements of $\Int C$.
If $\|z_m\|\to\infty$ as $m$ tends to infinity, then,
for all $x\in \Int C$,
$\Funk(x,z_m)$ also tends to infinity, and so $\Funk^+(x,z_m)=\Funk(x,z_m)$
for $m$ large enough. Then, the result of~\cite{walsh} shows that
$h=\lim_m h_m$ is of the form~\eqref{e-martinfunction}, for some $u\in C\setminus\{0\}$. It remains to examine
the case in which $z_m$ contains a bounded subsequence.
Then, possibly after replacing $z_m$ by a subsequence, we may assume
that $z_m$ converges to a point $u\in C$. Using
the fact established in~\cite{walsh} that for $x\in \Int C$,
$\Funk(x,y)$ is continuous in the second argument $y\in C$,
we deduce that $h_m$ converges to the function~\eqref{e-martinfunctionsfunkplus}. Finally, the characterisation of the horofunctions is straightforward.
\end{proof}

The metric $d_T$ associated to the Funk hemi-metric,
\[ d_T(x,y)=\max(\Funk(x,y),\Funk(y,x))
\]
is called the {\em Thompson metric}.
We recall that for this metric, the map \begin{equation}\label{geodesicThompson}\gamma_y(s)=\frac
{\beta^s-\alpha^s}{\beta\alpha^s-\alpha\beta^s+\beta^s-\alpha^s} \centerpoint{x}+\frac{\beta^s-\alpha^s}{\beta\alpha^s-\alpha\beta^s+\beta^s-\alpha^s} y \end{equation}
\noindent with $\beta=\exp(\Funk(\centerpoint{x},y))$ and $\alpha=\exp(-\Funk(y,\centerpoint{x}))$ is a
(straight line) geodesic joining $\centerpoint{x}$ to $y$, see~\cite{Nu}.
However, as shown in~\cite{walshnussbaum}, this choice of geodesics does not satisfy Busemann's non-positive
curvature condition
(even when $C=\RR_+^3$).
Actually, it is shown in~\cite{walshnussbaum} that
only a weaker inequality is valid. The latter inequality
is optimal (it is straightforward to find explicit
counter examples to the convexity inequality).

However, when $C$ is the cone $S_n^+$
of $n\times n$ positive semi-definite matrices, there is a different choice of geodesics
between any pair of matrices $Z,Y$ in the interior of $C$, involving geometric means:
\begin{align}
\gamma_Y(s)=Z^{\frac{1}{2}}\left(Z^{-\frac{1}{2}}YZ^{-\frac{1}{2}}\right)^sZ^{\frac{1}{2}} \enspace .
\label{e-geod}
\end{align}
See in particular~\cite{Nu,walshnussbaum}. The latter are indeed known to be geodesics in Thompson's and Hilbert's metric. %
Note that
when $Z$ is the identity matrix, this geodesic seen with logarithmic glasses
is a straight line: $\log \gamma_{Y}(s)=s\log Y$. The general form of the geodesic is obtained by considering the linear transformation $X\mapsto Z^{-1/2}XZ^{-1/2}$, which is an automorphism of the cone of positive definite matrices,
isometric both in Thompson's and Hilbert's metric,
and sending $Z$ to the identity.
This construction actually makes sense more generally when $C$ is a {\em symmetric cone}~\cite{faraut}, meaning that $C$ is self-dual ($C=C^\star$), and that the group of automorphisms leaving invariant the interior of $C$ acts transitively on this interior. Symmetric cones include in particular the Lorentz cones, the cones of real or complex positive semidefinite matrices, and Cartesian products of such cones. It is
shown in~\cite{gunawardenawalsh} that the metric convexity property of Definition~\ref{defistarshaped} holds for any symmetric cone, for the latter choice of geodesics, for any choice of the centre, both for the Thompson and Hilbert metric.
A similar observation was made by Lawson and Lim~\cite{lawsonlimmetric}.
Actually, the arguments in~\cite{gunawardenawalsh} apply as well to the Funk and $\Funk^+$ hemi-metrics.
Hence, the present results
apply to those self-maps of the interior of a symmetric cone that are non-expansive in any of the previously mentioned hemi-metrics.

\begin{remark}\label{rk-funk-symmetric}
When $C=S_n^+$ is the cone of positive semi-definite matrices,
when $x\in \Int C$, and $y\in C\setminus\{0\}$,
$M(y/x)$ coincides with the maximal eigenvalue of
the symmetric matrix $v:=x^{-1/2}yx^{-1/2}$. Moreover,
denoting by $w$ the (rank one) orthogonal projector on any eigenline
corresponding to this eigenvalue, one can check
that the linear form $\varphi(u):=\trace(x^{-1/2}wx^{-1/2}u)$
attains the supremum in~\eqref{e-def-funk}.
When $x\in C\setminus\{0\}$, with $y\leq \mu x$
for some $\mu>0$, a maximising linear form $\varphi$ can
be constructed in the same way, replacing $x^{-1/2}$ by
the square root of the Moore-Penrose inverse of $x$.
A dual argument allows one to construct a linear form
attaining the infimum in~\eqref{e-def-funkdual}.
Note also that these constructions can be easily adapted
to the case of a symmetric cone $C$.
\end{remark}

\begin{remark}\label{rk-psd}
In the case of the cone of symmetric positive definite matrices, the
Thompson, Hilbert, Funk, and $\Funk^+$ hemi-metrics all have the following form
\begin{align*}
\delta_{\nu}(A,B):= \nu(\log \Spec(A^{-1}B))=\nu(\log \lambda_1,\ldots,\log\lambda_n)
   \qquad A,B\in \Int S_n^+ \enspace,
\end{align*}
where $\nu$ is a symmetric function $\RR^n\to\RR$
(by {\em symmetric}, we mean that $\nu$ is invariant by permutation
of its variables), $\Spec(M)$ denotes the sequence of eigenvalues of a matrix $M$, counted
with multiplicities,
so that $\lambda_1, \cdots , \lambda_n$ denote the eigenvalues
of the matrix $A^{-1}B$ (which are real and positive since $A$ and $B$ are
positive definite). The previously listed hemi-metrics are obtained
by taking for $\nu$ the following maps, $\RR^n\to\RR$,
\[
\mu\mapsto
\max_{1\leq i\leq n}|\mu_i|,
\qquad\mu\mapsto
\max_{1\leq i,j\leq n}|\mu_i-\mu_j|,
\qquad\mu\mapsto
\max_{1\leq i\leq n}\mu_i,
\qquad\mu\mapsto
\max(\max_{1\leq i\leq n}\mu_i,0)\enspace,
\]
respectively. We may consider more generally any symmetric hemi-norm
$\nu$.
Then, the corresponding map $\delta_\nu$ can be checked to be a hemi-metric,
which is of Finsler type~\cite{nussbaumfinsler}.
Indeed, measuring the length of a vector $Y$ in the tangent space of $S_n^+$ at point $X$ by the hemi-norm $\nu(\Spec(X^{-1}Y))$, we get
\[
\delta_{\nu}(A,B)=\inf_\beta \int_0^1 \nu(\Spec(\beta(s)^{-1}\dot{\beta}(s)) ds \enspace,
\]
the infimum being taken over the set of curves $\beta$ from $[0,1]$ to
the interior of the cone, such that $\beta(0)=A$ and $\beta(1)=B$.
We recover as special cases of this construction
the invariant Riemannian metric~\cite{mostow} ($\nu$ is the Euclidean norm),  the log $p$-Schatten metric~\cite{friedlandfreitas} ($\nu$ is the $\ell_p$ norm), as well as metrics arising from symmetric Gauge functions~\cite{bhatia}
(the latter are symmetric hemi-norms that are invariant by a change of sign of each variable, and take positive values except at the origin).
For any choice of symmetric hemi-metric $\nu$, the map~\eqref{e-geod}
yields a family of geodesics for the hemi-metric $\delta_\nu$ joining a centre $Z$ to any matrix $Y$,
for which $\Int S_n^+$ is a metrically star shaped space.
This can be proved along the lines of~\cite{bhatia},
where the same property is established in the special
case in which $\nu$ is a symmetric Gauge function.
The reader may also consult~\cite{andruchowcorachstojanoff,gunawardenawalsh}, in which the significance of several matrix inequalities
in terms of non-positive curvature is brought to light.

A number of non-linear
maps, including discrete Riccati operators (Example~\ref{ex-riccati} below), are known to be non-expansive in various metrics $\delta_\nu$, see~\cite{bougerol,liverani,leelim,lawsonlimriccati}.
\end{remark}

\subsection{A Denjoy-Wolff type theorem}
We shall now consider specifically the case of order-preserving and positively
homogeneous maps. A self-map $T$ of $C$ is {\em order-preserving} if $x\leq y\implies T(x)\leq T(y)$. It is {\em positively homogeneous} (of degree one) if $T$ commutes with the product by a positive
constant. These notions make sense more generally if $T$ is only
defined on a subset of $C$ invariant by any such product.
It is readily checked that a self-map $T$
of the interior of $C$ is order-preserving and positively homogeneous if and only if it is non-expansive in the Funk metric.

\begin{corollary}\label{cor-symm}
Let $T$ be an order-preserving and positively homogeneous self-map
of the interior of a symmetric cone $C$, and let
\[
\overline{\rho}(T):=\inf_{y\in\Int C}\Funk(y,T(y)) \enspace .
\]
Then, for all $x\in \Int  C$,
\[
\overline{\rho}(T)=\rho(T):=\lim_{k\to\infty}\frac{\Funk(x,T^k(x))}{k}
\]
and there exists an extreme ray of $C$, such that for any element $w$ in this ray,
\[
\log\s(w,{T^k(x)}) \geq \log\s(w,x)+k\overline{\rho}(T).
\]
\end{corollary}
Here, $\s(\cdot,\cdot)$ denotes the scalar product of the Euclidean
space in which $C$ is self-dual.
\begin{proof}
The map $T$ is non-expansive in the Funk metric.
By Corollary~\ref{coregalite}, $\rho(T)=\overline{\rho}(T)$.
Let $h$ be the Martin function appearing in Theorem~\ref{maintheorem}.
Then, there exists a sequence of elements of $\Int  C$, $(z_m)_{m\geq 1}$
such that $i(z_m)$ converges to $h$. Let us now fix $x\in\Int  C$.
Since $C=C^\star$, any element $\varphi$ of an extreme ray of $C^\star$ can
be written as $\varphi(y)=\s(w,y)$ where $w$ is an element of an extreme
ray of $C$. Then, it follows from Lemma~\ref{lem-funk} that
for all $m$, there exists an element $w_m\in \Extr C$
such that
\[
\Funk(x,z_m)=\log\frac{\s(w_m,z_m)}{\s(w_m,x)}.
\]
A ray of $C$ is known to be extreme if and only if it contains a primitive idempotent of the associated Jordan algebra~\cite[Prop.~IV.3.2]{faraut}.
Moreover, the set of these primitive
idempotents is compact (Exercise~5 of Chap.~4 in~\cite{faraut}).
Since the previous property is invariant by a scaling of $w_m$, we may
assume that each $w_m$ is a primitive idempotent, and so, possibly after
extracting a subsequence, we may assume that $w_m$ converges to a primitive
idempotent $w$, which therefore belongs to an extreme ray of $C$.
For all $k\geq 0$, we get
\begin{align*}
\log \s(w,{T^k(x)})-\log\s(w,x)
&=\lim_m \log \s(w_m,{T^k(x)})-\log\s(w_m,x)\nonumber\\
&\geq \lim_m -\Funk(T^k(x),z_m)+\Funk(x,z_m)\nonumber\\
&=h(T^k(x))-h(x)\geq k\overline{\rho}(T)
\end{align*}
which concludes the proof.
\end{proof}

We next point out a variant of Corollary~\ref{cor-symm} concerning
order-preserving and sub-homogeneous self-maps of the
interior of a symmetric cone.
Recall that {\em sub-homogeneous} means
that for all $x\in \Int C$ and $\lambda\geq 1$, $T(\lambda x)\leq \lambda T(x)$.
Order-preserving and sub-homogeneous maps are easily seen
to be characterised by the following property
\[
\Funk^+(T(x),T(y))\leq \Funk^+(x,y)\enspace .
\]

\begin{corollary}\label{cor-symm2}
Let $T$ be an order-preserving and sub-homogeneous self-map
of the interior of a symmetric cone $C$, and let
\[
\overline{\rho}(T):=\inf_{y\in\Int C}\Funk^+(y,T(y)) \enspace .
\]
Then, for all $x\in \Int  C$,
\[
\overline{\rho}(T)=\rho(T):=\lim_{k\to\infty}\frac{\Funk^+(x,T^k(x))}{k}
\]
and, if $\overline{\rho}(T)>0$, there exists an extreme ray of $C$, such that for any element $w$ in this ray,
\[
\log\s(w,{T^k(x)}) \geq \log\s(w,x)+k\overline{\rho}(T).
\]
\end{corollary}
\begin{proof}
Since the hemi-metric $\Funk^+$ is bounded from below,
the observation made at the end of the proof of Theorem~\ref{maintheorem}
shows that if $z_m$ is any sequence of points of $\Int C$ such that $i(z_m)$ converges to the Martin function $h$ in this theorem, then, the sequence $z_m$ must be such that $\|z_m\|\to\infty$. Then,
it follows from the characterisation of the boundary of the $\Funk^+$ metric
(proof of Proposition~\ref{prop-walsh-funkp})
that $h$, which cannot be of the form~\eqref{e-martinfunctionsfunkplus},
is necessarily
of the form~\eqref{e-martinfunction}.
Therefore, we conclude as in the proof of Corollary~\ref{cor-symm}.
\end{proof}
\begin{remark}
Corollary~\ref{cor-symm} should be compared with a related result of Gunawardena and Walsh~\cite[Th.~1]{gunawardenawalsh}, in which the map $T$ is only required to be non-expansive in Thompson's metric. Then, it is shown that all the accumulation points of the sequence
$(\log T^k(x))/k$ belong to the same face of a ball which respect to the norm
in the tangent space of $C$ at point $x$.
\end{remark}
\begin{remark}
Lins established a Denjoy-Wolff theorem of a different nature in~\cite{lins},
concerning self-maps of a polyhedral domain that are non-expansive in Hilbert's
metric.
\end{remark}
\subsection{The Collatz-Wielandt theorem}
We now show that the Collatz-Wielandt formula~\eqref{e-cw}
is a special case of the maximin formula~\eqref{e-minimax}.

If $T$ is an order preserving and positively sub-homogeneous self-map of the interior of a (closed, convex and pointed) cone $C$ in a finite dimensional
vector space,  we define the {\em radial extension} $\hat{T}$ of $T$
to the closed cone $C$
by
\[
\hat{T}(x):=\lim_{\epsilon\to 0^+} T(x+\epsilon z),
\qquad \forall x\in C\enspace,
\]
where $z$ is any point in the interior of $C$. One verifies that this definition is independent of the choice of $z$; equivalently, we may define $\hat{T}(x)$ as $\inf \{T(y)\mid y-x\in \Int C\}$. The order preserving and
positively (sub-)homogeneous character of $T$ carries over to $\hat{T}$.
It is shown in~\cite{MR1960046} that $T$ does not necessarily have
a continuous extension to $C$, but that it always does
if $C$ is polyhedral.
\begin{lem}\label{lem-ourcw}
Let $T$ be an
order-preserving and positively
sub-homogeneous self-map of the interior of
a (convex closed and pointed) cone $C$
in a finite dimensional vector space,
and consider the Martin function $h$ in~\eqref{e-martinfunction}
associated to a vector $u\in C\setminus\{0\}$. Then, for all
$\lambda\in\RR$,
\begin{align}
\big(\forall x\in \Int C, \; h(T(x))-h(x)\geq \lambda \big)\iff
\big(\forall \gamma>0,\; \hat{T}(\gamma u)\geq \gamma \exp(\lambda)u\big)\enspace .
\label{e-tech}
\end{align}
\end{lem}
\begin{proof}
Since, $h(T(x))-h(x)=-\Funk(T(x),u)+\Funk(x,u)$, the left-hand side condition
in~\eqref{e-tech} holds if and only if
\begin{align}
u\leq \exp\big(-\lambda+\Funk(x,u)\big)T(x),\qquad \forall x\in \Int C \enspace .
\label{e-myrw}
\end{align}
For $\epsilon>0$ and $\gamma>0$, define $x_\epsilon:=\gamma u+\epsilon z$, where
$z$ is a given element in $\Int C$.
Then,
$\Funk(x_\epsilon, u)\leq \Funk(\gamma u, u)=-\log \gamma $.
Taking $x=x_\epsilon$ in~\eqref{e-myrw},
and letting $\epsilon$ tend to zero,
we get
\begin{align*}
u\leq \exp(-\lambda)\gamma^{-1}\hat{T}(\gamma u) \enspace ,
\end{align*}
and so, the condition at the right-hand side of~\eqref{e-tech}
holds. Conversely, fix $x\in \Int C$, and
choose $\gamma$ such that $\gamma\Funk(x,u)\geq 1$. We have
\begin{align*}
\gamma u &\leq \exp(-\lambda) \hat{T}(\gamma u) \qquad 
\text{(using the right-hand side of~\eqref{e-tech})}\\
&\leq  \exp(-\lambda)
\hat{T}(\gamma\exp(\Funk(x,u))x)
\qquad \text{(since $\hat{T}$ is order preserving)}\\
&\leq  \exp(-\lambda)
\gamma\exp(\Funk(x,u))\hat{T}(x)
\qquad \text{(since $\hat{T}$ is sub-homogeneous)}\\
&=
\gamma \exp(-\lambda + \Funk(x,u))T(x) \enspace .
\end{align*}
Cancelling $\gamma$, we arrive at~\eqref{e-myrw}.
\end{proof}

Observe that when $T$ is positively sub-homogeneous,
the following {\em recession map}
\begin{align}\label{e-def-recession}
\hat{T}_r:C\to C,\qquad \hat{T}_r(x)=\lim_{\gamma\to\infty}
\gamma^{-1}{\hat{T}(\gamma x)}
=\inf_{\gamma>0}\gamma^{-1}{\hat{T}(\gamma x)}
\end{align}
is well defined. Of course, $\hat{T}=\hat{T}_r$ if $T$ is positively
homogeneous.
We now arrive at the following corollary of Theorem~\ref{th-minimax},
which extends the Collatz-Wielandt formula~\eqref{e-cw}.
\begin{corollary}\label{th-ourcw}
Let $T$ be an
order-preserving  and positively sub-homogeneous self-map
of the interior of
 a symmetric cone in a finite dimensional vector space,
and assume that
\begin{align}\label{e-def-escapefunk}
\rho(T):=\lim_{k\to\infty}\frac{\log\Funk(x,T^k(x))}{k} >0 \enspace .
\end{align}
Then,
\begin{align}\label{e-cw-gen}
\inf_{y\in \Int C}
\max_{w\in \Extr C}\log\frac{\s(w,{T(y)})}{\s(w,y)}
=\rho(T)=
\max_{u\in C\setminus\{0\}}\min_{\scriptstyle w\in \Extr C\atop\scriptstyle  \s(w,u)\neq 0}
 \log \frac{\s(w,{\hat{T}_r(u)})}{\s(w,u)} \enspace .
\end{align}
Moreover, when $T$ is positively homogeneous, the same
conclusion remains valid even when the condition $\rho(T)>0$
does not hold.
\end{corollary}
\begin{proof}
We apply Theorem~\ref{maintheorem} to the space $X:=\Int C$ equipped with the hemi-metric $\delta:=\Funk^+$. Let us denote by $\rho_+(T)$ the escape
rate of $T$ in this hemi-metric, so that
\begin{align}\label{e-cw-rhoplus}
\rho_+(T)=\inf_{y\in\Int C}\Funk^+(y,T(y)) \enspace .
\end{align}
Since $\Funk^+=\max(\Funk,0)$, the hemi-metrics
$\Funk$ and $\Funk^+$ coincide as soon as $\Funk^+$ or $\Funk$ is positive. Hence, $\rho(T)>0$ implies that $\rho_+(T)=\rho(T)$ and that the term $\Funk^+(y,T(y))$ in~\eqref{e-cw-rhoplus} can be replaced by $\Funk(y,T(y))$.
Together with~\eqref{e-def-funk}, this gives the first equality in~\eqref{e-cw-gen}.

We observed in the proof of Corollary~\ref{cor-symm2} that when
the escape rate in the $\Funk^+$ hemi-metric is positive, the corresponding
horofunction $h$ in Theorem~\ref{maintheorem} is necessarily
of the form~\eqref{e-martinfunction}.
Since the limit coincides with the infimum in the definition of the recession function (see \eqref{e-def-recession}), the right-hand side
condition in~\eqref{e-tech} is equivalent to
\[
\hat{T}_r(u)\geq \exp(\lambda) u \enspace.
\]
Then, using Lemma~\ref{lem-ourcw}, we rewrite the maximin characterisation
in Theorem~\ref{maintheorem} as
\begin{align*}
\rho(T)&=\max_{h}\inf_{x\in \Int C}h(T(x))-h(x)\nonumber\\
&= \max\{\lambda\mid \exists u\in C\setminus\{0\},\; \hat{T}_r(u)\geq \exp(\lambda)u \}
\nonumber
\\
&=\max_{u\in C\setminus\{0\}}
\log \inf_{w\in \Extr C\;\s(w,u)\neq 0}\frac{\s(w,{\hat{T}_r(u)})}{\s(w,u)}
\qquad \text{(by~\eqref{e-def-funkdual}).}
\end{align*}
Moreover, it follows from Remark~\ref{rk-funk-symmetric} that the latter infimum is attained.

Finally, when $T$ is positively homogeneous, we
equip the space $\Int C$ with the hemi-metric $\Funk$
instead of $\Funk^+$. Then, the first equality in~\eqref{e-cw-gen} follows readily from Theorem~\ref{maintheorem} and~\eqref{e-def-funk} (the assumption that $\rho(T)>0$
is not needed any more), and the proof of the second inequality
is unchanged.
\end{proof}

\begin{remark}\label{cara-handy}
The characterisation~\eqref{e-cw-gen} can be rewritten equivalently as
\begin{align}
\rho(T)&=
\log\inf\{ \mu >0 \mid \exists y\in \Int C,\; T(y)\leq \mu y\}\label{e-new1}\\
& = \log\max\{ \mu \geq 0\mid \exists u\in  C\setminus\{0\},\; \hat{T}_r(u)\geq \mu u\}
\label{e-new2}
\enspace .
\end{align}
\end{remark}
\begin{remark}
When $C$ is the standard positive cone, and
$T$ is a {\em continuous} order preserving and positively homogeneous
self-map of the {\em closed} cone $C$, Nussbaum~\cite{nussbaum86} showed
that
\begin{align*}
\rho(T)=
 \log\max\{ \mu \geq 0\mid \exists u\in  C\setminus\{0\},\; {T}(u)= \mu u\}
\enspace,
\end{align*}
which is more accurate than~\eqref{e-new2}.
Moreover, some of the results
of~\cite{nussbaum86} have been extended in~\cite{Nuss-Mallet}
to more general (normal, possibly infinite dimensional) cones.
However, Corollary~\ref{th-ourcw} is
applicable when $T$ is only defined on the {\em interior} of the cone and does not extend continuously to its boundary (this situation does occur in practice, see the discussion in~\cite{MR1960046} of the matrix harmonic mean).
\end{remark}

\begin{example}\label{ex-riccati}
We finally illustrate the previous Denjoy-Wolff type results with an example originating from quadratic optimal control.
Let $A,B$ be $n\times n$ positive semidefinite matrices, and let $M$ be a
$n\times n$ invertible matrix. Consider the following
discrete time Riccati operator, which is a self-map of
the interior of the cone $S_n^+$ of $n\times n$ positive semidefinite symmetric
matrices,
\[
T(X)=A+M(B+X^{-1})^{-1}M^* \enspace ,
\]
where $M^*$ denotes the transpose of $M$.
We refer the reader to~\cite{bougerol,liverani,leelim} for more background on these maps. In particular, $T$ is order-preserving,
and, as shown by Liverani and Wojtkowski~\cite{liverani}, $T$ is non-expansive in the Thompson metric. We noted above that the latter property,
together with the preservation of order,
implies that $T$ is positively sub-homogeneous, and so, Corollaries~\ref{cor-symm2} and~\ref{th-ourcw} apply to the map $T$.

If the matrices $A$ and $B$ are positive definite, then $T$ is known to be a strict contraction in Thompson metric~\cite{liverani}. It follows that $T$ has a unique fixed point in $\Int S_n^+$ to which
every orbit converges. In particular, the escape rate $\rho(T)$ is zero.

We next examine a more interesting case in which $B$ is of rank one, leading
to a positive escape rate. Assume that $n=2$, write $B=vv^*$
for some (column) vector $v$, and assume that $M$ is of the form $M=\alpha I$, with $\alpha>1$, where $I$ is the identity matrix.

We claim that $\rho(T)=2\log\alpha$. To see this,
let $u$ denote a non-zero vector orthogonal to $v$, and consider
$U=uu^*$. Then, it is easily checked that the radial extension of $T$
satisfies $\hat{T}(\gamma U)=A+\alpha^2 \gamma U$
for all $\gamma>0$, and so, $\hat{T}_r(U)=\alpha^2U$.
Using~\eqref{e-new2}, we deduce that $\rho(T)\geq 2\log \alpha$.
To show that the equality holds, we observe that $T(X)\leq A+MXM^*$.
Hence, for all $s>0$, $T(sI)\leq A +\alpha^2 sI \leq (\lambda_1(A)s^{-1}+\alpha^2 ) sI$, where $\lambda_1(A)$ denotes the maximal eigenvalue of $A$,
which by~\eqref{e-new1}, implies that $\rho(T)\leq \log(\lambda_1(A)s^{-1}+\alpha^2)$.
Letting $s$ tend to infinity, we arrive at $\rho(T)\leq 2\log\alpha$,
and so $\rho(T)=2\log\alpha$.

A possible choice of the horofunction $h$ appearing in Theorem~\ref{maintheorem}
is
\[h(X)=-\Funk(X,U)+\Funk(I,U)\enspace,
\]
where the basepoint is the identity matrix $I$. Then, the linear form $x\mapsto\s(w,x)$ constructed in the proof of Corollary~\ref{cor-symm2} can be checked to be $X\mapsto \s(U,X)$ (or any scalar multiple of this form).

For $n=2$, the cone of positive semidefinite matrices
$S_n^+$ has the shape of a Lorentz (ice-cream) cone.
Then, the superlevel set $h(X)\geq \lambda$ is nothing but the interior of
the translated Lorentz cone $\exp(\lambda-\Funk(I,U)) U+ S_2^+$.
This is illustrated in Figure~\ref{fig-denjoy}.

\end{example}

\begin{figure}[htbp]%
\begin{center}\includegraphics[scale=0.8]{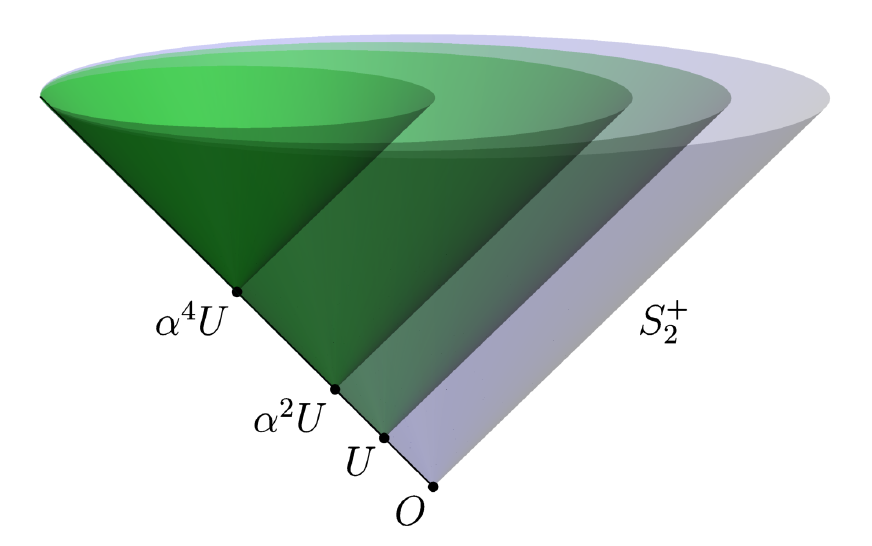}\end{center}
\caption{Theorem~\ref{maintheorem} applied to a discrete Riccati operator on the cone of positive definite matrices $\Int S_2^+$ (Example~\ref{ex-riccati}). The horoballs (superlevel sets) of the horofunction $h$ are nested cones (in green) intersecting the cone $S_2^+$ (in blue) along the ray generated by the matrix $U$.
The map $T$ sends every horoball $h\geq \lambda$ to the horoball $h\geq \lambda+\rho(T)$. Three such horoballs are represented, with apices $U$, $\alpha^2U$, $\alpha^4U$, corresponding to $\lambda=0,\rho(T),2\rho(T)$, respectively.}
\label{fig-denjoy}
\end{figure}

\paragraph*{\hspace*{\parindent}Acknowledgements.}
 This work originates from discussions
of Brian Lins and the first author
at the AIM workshop on non-negative matrix theory in Palo Alto in 2008,
where they jointly obtained an early
version of Theorem~\ref{maintheorem} concerning the special Banach space case. The first author thanks Brian Lins for these discussions, and also for
more recent observations. 

\end{document}